\documentclass[11pt]{article}
\usepackage[utf8]{inputenc}
 
\title{Differentiating the State Evaluation Map from Matrices to Functions on Projective Space}
\author{Ghaliah Alhamzi$^1$ \& Edwin Beggs$^2$ }
\date{%
	$^1$Department of Mathematics and Statistics, College of Science, \\ Imam Mohammad Ibn Saud Islamic University (IMSIU), Riyadh, Saudi Arabia \\
		$^2$College of Science, Swansea University, Wales 
	\\[2ex]%
}

\usepackage{yfonts}
\usepackage{amsmath}
\usepackage{amsfonts}
\usepackage{amssymb}
\usepackage{amsthm}
\usepackage{mathrsfs}
\usepackage{bm}
\usepackage[pdftex]{graphicx,color}
\usepackage{graphics,color}
\usepackage{epstopdf}
\usepackage{graphicx}
\usepackage{textcomp}
\usepackage{pgfplots, pgfplotstable}
\usepackage[font=small,labelfont=bf,labelsep=colon]{caption}
\usepackage{pdfpages}
\usepackage{fullpage}
\usepackage{calligra} 
\usepackage[all]{xypic}
\usepackage{float}
\usepackage{subcaption} 
\usepackage[normalem]{ulem}
 \usepackage{amsmath,amsfonts,amsthm,amssymb}
\usepackage{cancel}

\newtheorem{theorem}{Theorem}[section]
\newtheorem{lemma}[theorem]{Lemma}
\newtheorem{cor}[theorem]{Corollary}
\newtheorem{proposition}[theorem]{Proposition}

\theoremstyle{remark}

\newtheorem{definition}[theorem]{Definition}

\theoremstyle{definition}

\newcommand{\id}{\mathrm{id}}

\newcommand{\extd}{\mathrm{d}}
\newcommand{\tens}{\mathop{\otimes}}

\newcommand{\C}{\mathbb{C}}

\newcommand{\mathsym}[1]{{}}
\newcommand{\unicode}[1]{{}}

\newcommand{\di}{\,\mathrm{d}}

\begin{document}
 
\maketitle
 
 \abstract{We show that the pure state evaluation map from $ M_{n}(\C) $ to $ C(\C \mathbb{P}^{n-1}) $ (a completely positive map of $ C^{*} $-algebras) extends to a cochain map from the universal calculus on $ M_{n}(\C ) $ to the holomorphic $ \bar{\partial} $ calculus on $ \C \mathbb{P}^{n-1} $. The method uses connections on Hilbert $ C^{*} $-bimodules. This implies the existence of various functors, including one from $ M_{n}(\C) $ modules to holomorphic bundles on $ \C \mathbb{P}^{n-1} $. }
 	
 \section{Introduction}
 For a subset $ X $ of the state space $ S $ of a $ C^{*} $-algebra $ A $ we have a positive “state evaluation map” $ \delta : A \rightarrow C(X) $ given by $ \delta(a)(\phi) =\phi(a)$ for $ a\in A $ and $ \phi \in  X$. For $ M_{n }(\C) $ the result of Choi \cite{Choi} gave  the pure state space as $  \C \mathbb{P}^{n-1} $. We use the KSGNS construction \cite{Lance} to analyse the case $ A=M_{n}(\C) $ and $ X=\C \mathbb{P}^{n-1} $ and then consider the differentiability of the state evaluation map. To do this we begin by constructing the Hilbert $ C^{*} $-bimodule giving the state evaluation map. Then we use the methods of connections on bimodules to connect the differential structure on $ M_{n}(\C) $ (we take the universal calculus ) to that on $ \C \mathbb{P}^{n-1} $ (the usual calculus). Here we follow the methods in \cite{BMbook}, but then find that conditions required there do not apply, so in Section \ref{Se3} we consider a more general theory extending the results in \cite{BMbook}. As a result Proposition \ref{prop12} on an induced functor from left $ M_{n} $-modules to holomorphic bundles on $  \C \mathbb{P}^{n-1} $ is phrased in terms of holomorphic bundles rather than flat bundles on $  \C \mathbb{P}^{n-1} $. (For brevity we often refer to $ M_{n}(\C) $ just as $ M_{n} $.) Also our main result Theorem \ref{thm2} on extending the state evaluation map to a cochain map uses the $ \bar{\partial} $ calculus on projective space.
 
 The main reason why we chose to do this construction with $ M_{n} $ is the concrete construction of the state space. More generally, it might be possible to put a differential structure on the pure state space of a $ C^{*} $-algebra, even if we know little about the state space. For this one thing to remember is that there is a very general idea of calculus on infinite dimensional spaces \cite{Milnor} using directional derivatives. 
 
 Apart from the concrete description of the state space, another reason why we are interested in the calculi on matrix algebras and the link with representations and states is Connes' noncommutative derivation of the standard model \cite{AliAlConnes}. The fact that from a relatively simple noncommutative beginning involving matrices Connes constructs the standard model indicates that there probably something very interesting in the geometry of the initial noncommutative space. Most gauge  theories in physics are described in terms of calculi and so we are naturally led to questions about calculi on matrices and how they relate to states. 
 
 We use the notation that $ h_{i}\in \mathrm{Col}^{n}(\C) $ is the column vector with $ 1 $ in position $ i $ and zero elsewhere, and that $ E_{ij}\in M_{n}(\C) $ in the matrix with 1 in row $ i $ and column $ j $ and zero elsewhere. An element of $ \C \mathbb{P}^{n-1} $ is written in homogenous coordinates as $ [(v_{1} \dots v_{n})] $ where we suppose $ \sum|v_{i}|^{2}=1 $. We sum over repeated indices unless otherwise indicated.
 \section{Preliminaries } \label{Se1} 
 \subsection{Calculi and connections}
 \begin{definition}
 	Given a first order calculus $ (\Omega_{A}^{1} , \di ) $ on an algebra $ A $, the maximal prolongation calculus $ \Omega_{A} $ has relations $\sum\di c_{i} \wedge \di a_{i}=0 $ for every relation $ \sum c_{i}  \di a_{i}=0 $ on $ \Omega_{A}^{1} $, where $ c_{i}, a_{i} \in A $.
 \end{definition}
 \begin{definition}
 	The universal first order calculus $ \Omega^{1}_{\mathrm{uni}}(A) $ on a unital algebra $ A $ is defined by 
 	\[\Omega^{1}_{\mathrm{uni}}(A) = \ker \cdot : A \tens A \rightarrow  A\]
 	where $ \cdot $ is the algebra product and $ \di _{\mathrm{uni} }a =1 \tens a-a\tens 1 $.
 \end{definition}
The maximal prolongation of the universal calculus has $ \Omega_{\mathrm{uni}}^{n}(A) \subset A^{\tens  n+1} $ which is the intersection of all the kernels of the multiplication maps between neighbouring factors, i.e. 
\[\Omega_{\mathrm{uni}}^{2}(A)=\ker (\cdot \tens \id : A \tens  A \tens A  \rightarrow A \tens A ) \cap  \ker (\id\tens \cdot  : A \tens  A \tens A \rightarrow A \tens A ) \ . \]
We now assume that the unital algebras $ A $ and $ B $ have calculi $ \Omega_{A}^{n} $ and $ \Omega_{B}^{n} $ respectively. 
\begin{definition}\label{de2}
	A right connection $ \nabla_{E}:E\rightarrow E \tens_{B} \Omega_{B}^{1} $ on a right $ B $-module $ E $ is a linear map obeying the right Leibniz rule for $ e\in E $ and $ b\in B $
	\begin{equation}\label{nablaE}
		\nabla_{E} (e \,b)=  e \tens \extd b+\nabla_{E}\,(e).b\ .
	\end{equation}
\end{definition}
\begin{definition}
	Given the right connection $ (E,\nabla_{E}) $ in Definition \ref{de2} we define for $ n \geq 1 $
	\begin{equation*}
		\nabla_{E}^{[n]}: E \tens_{B} \Omega_{B}^{n} \rightarrow E \tens_{B} \Omega_{B}^{n+1}
	\end{equation*}
	by $ \nabla ^{[1]} _{E}=\nabla_{E}  $ and for $ n \geq 2 $
	\begin{equation*}
		\nabla_{E}^{[n]}(e \tens \xi)=\nabla_{E} \, e \wedge \xi + e \tens \di \xi \ .
	\end{equation*}
	The curvature of $ E $ is the right bimodule map 
	\[R_{E}= \nabla_{E}^{[1]} \nabla_{E} : E \rightarrow E \tens_{B} \Omega_{B}^{2}\]
	and then for $ e \tens \xi \in E \tens_{B} \Omega_{B}^{n} $
	\begin{equation*}
		\nabla_{E}^{[n+1]} \nabla_{E}^{[n]}(e \tens \xi)=R_{E} (e )\wedge \xi \ .
	\end{equation*}
\end{definition}
The idea of a bimodule connection was introduced in \cite{DVMic}, \cite{DVMass} and \cite{Mourad} and used in \cite{FioMad}, \cite{Madore}. It was used to construct connections on tensor products in \cite{BDMS} (see Proposition \ref{prop4}). 
\begin{definition}
	If $ E $ is an $ A $-$ B $ bimodule then $ (\nabla_{E}, \sigma_{E}) $ is a right bimodule connection where $ \nabla_{E} $ is a right connection and there is a bimodule map
	\[\sigma_{E}: \Omega_{A}^{1}\tens_{A} E \rightarrow E\tens_{B} \Omega_{B}^{1}\]
	so that 
	\begin{equation*}
		\nabla_{E}(a e)= \sigma_{E}(\extd a \tens e)+a. \nabla_{E}\, e\ .
	\end{equation*}
\end{definition}
\subsection{ Hilbert bimodules}
Note that, unlike most of the literature on Hilbert $ C^{*}$-modules, we explicitly use conjugate bundles and modules.  This is required to make the usual tensor products and connections work with inner products. Suppose that $ A $ and $ B $ are $ * $-algebras. For a left $ A $-module $ E, \bar{E} $ is the conjugate vector space with right $ A $-action $ \overline{e}.a=  \overline{a^{*}e}$, and for a right $ A $ module $ F $, $ \bar{F} $ is the conjugate vector space with left $ A $-action $ a. \overline{f}=\overline{f.a^{*}} $. For our $ A $-$ B $ module $ E $, $ \bar{E} $ is a $ B $-$ A $ bimodule with $ b \bar{e}= \overline{e b^{*}}$ and $ \bar{e} a=\overline{a^{*} e}$. 
\begin{definition}
	A differential calculus $ (\Omega_{A}, \di ) $ on a $ * $-algebra $ A $ is a $ * $-differential calculus if there are antilinear  operators $ *: \Omega_{A}^{n} \rightarrow  \Omega_{A}^{n} $ so that $ (\xi \wedge \eta)^{*}=(-1)^{|\xi||\eta|} \eta^{*} \wedge \xi^{*}$ where $ |\xi| $ is the degree of $ \eta $, i.e. $ \eta \in \Omega_{A}^{|\eta|} $ and $ (\di \xi) ^{*}=\di (\xi^{*})$.
\end{definition}
We now suppose that $ A $ and $ B $ have $ * $-calculi. Then for our right bimodule connection $ (\nabla_{E}, \sigma_{E}) $ we have a corresponding left bimodule connection $ (\nabla_{\bar{E}}, \sigma_{\bar{E}}) $ on $ \bar{E} $ given by $ \nabla_{\bar{E}} \bar{e}=\xi^{*} \tens \bar{f}$ where $ \nabla_{E} e=f\tens \xi $ (sum implicit) and $ \sigma_{\bar{E}}(\bar{e} \tens \eta)=k^{*}\tens \bar{g} $ where $ \sigma_{E}(\eta^{*} \tens e)=g\tens k $.

We give a definition of inner product on an $ A $-$ B $ bimodule $ E $, where $ A $ and $ B $ are $  *$-algebras. This is taken from the definition of  Hilbert bimodules in \cite{Lance}, omitting norms and completion as we will need smooth function algebras. Of course,  the  modules with inner product we will talk about have completions which really are Hilbert bimodules.
\begin{definition}
	$ A \, B $-valued inner product on an $A  $-$ B $ bimodule $ E $ is a $ B $-bimodule map $ \langle , \rangle: \bar{E} \tens_{A} E \rightarrow B $ obeying $ \langle\bar{e'}, e\rangle ^{*}=\langle\bar{e}, \bar{e'}\rangle$ for all $ e',e \in E $ (the Hermitian condition) and $ \langle \bar{e},e\rangle  \geq 0$ and $ \langle  \bar{e},e\rangle  = 0$ only where $ e=0 $.
\end{definition}
Given an inner product $ \langle,\rangle : \bar{E}\tens_{A} E\rightarrow B$ the right connection $ \nabla_{E}  $ preserves the inner product if 
\begin{equation}\label{eqn41}
	(\id \tens  \langle,\rangle )(\nabla_{\bar{E}} \tens \id)+(\langle,\rangle \tens \id)(\id \tens \nabla_{E})=\extd \langle,\rangle \ .
\end{equation}
 \subsection{Line bundles and calculus in  $ \C \mathbb{P}^{n-1} $ } \label{subsec1}
On $ \C \mathbb{P}^{n-1} $ we have homogenous coordinates $ v_{i}\in \C $ for $ 1\leq i \leq n $. We take $ \underline{v}=(v_{1},\dots, v_{n}) $ to lie on the sphere $ S^{2n-1} $ in $ \C^{n} $, i.e. $ \sum_{i} v_{i} \bar{v}_{i}=1 $. There is an action of the unit norm complex numbers $ U_{1} $ on $ S^{2n-1} $ by 
\begin{align*}
	z \triangleright (v_{1}, \dots, v_{n})= (z v_{1}, \dots ,zv_{n})\ .
\end{align*}
We define $ \mathbb{C} \mathbb{P}^{n-1}$ as $ S^{2n-1} $ quotiented by this circle action, identifying points $ z\triangleright \underline{v} \cong \underline{v}$ for all $ z \in U_{1} $. We use notation $ [\underline{v}] \in \C \mathbb{P}^{n-1}$ for the equivalence classes. We consider subsets of continuous functions on $ S^{2n-1} $, defining for integer $ m $
\begin{align*}
	C_{m}(\C \mathbb{P}^{n-1})= \{f \in C(S^{2n-1}): f(z \triangleright \underline{v})=z^{m} f(\underline{v})\quad \text{for all} \ z \in U_{1}, \underline{v}\in S^{2n-1}\}
\end{align*}
and similarly $ C^{\infty} _{m}(\C \mathbb{P}^{n-1}) $ to be smooth functions. Then $ C^{\infty} _{0}(\C \mathbb{P}^{n-1}) $ is the usual smooth functions on $ \C \mathbb{P}^{n-1} $.
There is an alternative view given by grading monomials in $ v_{i} $ and $ \bar{v}_{i} $ by $ ||v_{i} ||=1 $ and $|| \bar{v}_{i}||=-1 $. Then a monomial of grade $ m $ is in $ C _{m}(\C \mathbb{P}^{n-1}) $.  
A grade zero monomial such as $ v_{1}  \bar{v}_{2} \,\bar{ v}_{3} v_{4}   $ is invariant for the circle action and so gives a function on $ \C \mathbb{P}^{n-1} $. 

 An element of the tautological bundle $ \tau $ at $ [\underline{v}]\in \C \mathbb{P}^{n-1} $ is given by $ \alpha \underline{v} \in \C^{n} $ for $ \alpha \in \C $ and the inner product on $ \tau $ is given by 
\begin{equation}\label{underline}
	\langle \overline{\alpha \,\underline{v}}, \beta \underline{v} \rangle= \bar{\alpha}\beta \in \C\ ,
\end{equation} 
noting the use of the conjugate bundle to give bilinearity and be consistent with the earlier Hilbert $ C^{*} $-bimodule inner product. A section of the tautological bundle is a function $ r:\C \mathbb{P}^{n-1 } \rightarrow \mathrm{Row}^{n}(\C)$ so that $ r\big([\underline{v}]\big) $ is a multiple of $ \underline{v} $. We have a $ 1-1 $ correspondence between continuous sections of $ \tau $ and $ C_{-1}(\C \mathbb{P}^{n-1})  $. If $ f\in C_{-1}(\C \mathbb{P}^{n-1})  $ then $ (f. v_{1}, \dots , fv_{n}) $ is a section and if $( r_{1}, \dots, r_{n}) $ is a section then $ r_{i} \bar{v}^{i} $ is in $ C_{-1}(\C \mathbb{P}^{n-1})  $.    

Recalling that $ \sum_{i} v_{i} \bar{v}_{i}=1 $ and applying $ \di $ gives $ \sum_{i} \big(\di v_{i} \,\bar{v}_{i}+v_{i} \di \bar{v}_{i}\big)=0 $ and as we require a complex calculus on $ \C \mathbb{P}^{n-1} $ we get both $ \sum_{i} \di v_{i}\, \bar{v}_{i}=0 $ and $ \sum_{i} v_{i} \di \bar{v}_{i}=0 $ as relations on $ \Omega^{1}(\C \mathbb{P}^{n-1}) $. Applying $ \di $ again gives $ \sum_{i} \di v_{i} \wedge \di \bar{v}_{i}=0 $ in $ \Omega^{2}(\C \mathbb{P}^{n-1}) $.
\subsection{Categories of modules and connections}\label{subsec3}
For an algebra $ A $ we take $ \mathcal{M}_{A} $ to be the category of right $ A $-modules and right module maps. If $ A $ has a differential calculus we take $ \mathcal{E}_{A} $ to be the category with objects $ (E, \nabla_{E}) $ where $ E $ is a right $ A $-module and $ \nabla_{E} $ is a right connection on $ E $. A morphism $ T $ from $( E, \nabla_{E}) $ to be $ (F,\nabla_{F}) $ consists of a right module map $ T: E \rightarrow F $ which commutes with the connections, i.e.
\[\nabla_{F} T= (T \tens \id) \nabla_{E} : E \rightarrow F\tens_{A} \Omega_{A}^{1}\ .\]
\begin{proposition}\label{prop4}
	For a right $ A $-$ B $ bimodule connection $ (\nabla_{W}, \sigma_{W}) $, there is a functor $ \tens_{A} W: \mathcal{E}_{A} \rightarrow \mathcal{E}_{B}$ sending $ (\nabla_{F}, F) $ to $ (\nabla_{F \tens W}, F \tens_{A} W) $, where $ \nabla_{F \tens W}  $ is 
	\begin{equation*}
		\nabla_{F \tens W} (f\tens e)=(\id \tens \sigma_{W})(\nabla_{F}(f) \tens e)+f \tens \nabla_{W}(e)\ .
	\end{equation*}
\end{proposition}
\subsection{Holomorphic bundles }\label{subsec2}
Let $ B $ be a $ * $-algebra with a $ * $-differential calculus. We use the noncommutative complex calculi from \cite{KhLavS}, \cite{BegSmiComplex}. Suppose we have a direct sum decomposition $ \Omega_{B}^{n}=\oplus _{p+q=n} \Omega_{B}^{p,q} $ as bimodules, and that $ \Omega_{B}^{p,q} \wedge \Omega_{B}^{s,t} \subset \Omega_{B}^{p+s,q+t} $ and $ \di \Omega_{B}^{p,q} \subset \Omega_{B}^{p+1,q} \oplus \Omega_{B}^{p,q+1} $ and $ (\Omega^{p,q})^{*}=\Omega^{q,p} $.  Using the projection operations for the direct sum $ \pi^{p,q}: \Omega_{B}
^{p+q} \rightarrow \Omega_{B}^{p,q}$ we can define 
\begin{gather*}
	\partial=  \pi^{p+1,q} \di  :  \Omega_{B}^{p,q} \rightarrow  \Omega_{B}^{p+1,q} \cr
	\bar{\partial}=  \pi^{p,q+1} \di  :  \Omega_{B}^{p,q} \rightarrow  \Omega_{B}^{p,q+1} 
\end{gather*}
which gives a holomorphic calculus. Given a right connection $ \nabla_{G}: G \rightarrow G \tens_{B} \Omega_{B}^{1} $ then we define $ \bar{\partial}_{G}=(\id \tens  \pi^{0,1}) \nabla_{G}:G \rightarrow  G \tens_{B} \Omega_{B}^{0,1}$. The holomorphic curvature of $ G $ is defined to be the curvature of the $ \bar{\partial}_{G} $ connection, i.e. 
\[(\id \tens \bar{\partial}+\bar{\partial}_{G} \wedge \id) \bar{\partial}: G \rightarrow G \tens \Omega_{B}^{0,2}\ .\]
\begin{definition}\label{de3}
	Suppose that we have a right connection $ \bar{\partial}_{G}: G\rightarrow G \tens_{B} \Omega_{B}^{0,1}$ with holomorphic curvature zero. Then $ (G, \bar{\partial}_{G}) $ is called a holomorphic right module.
\end{definition}
	\section{The KSGNS construction of the state evaluation map}\label{ch1}
	For a subset $ X\subset S $ of the state spaces of a $ C^{*} $-algebra $ A $ the positive map $ \delta: A\rightarrow C(X) $ is given by $ \delta(a)(\phi)=\phi (a) $ for $ a\in A $ and $ \phi \in X $. We use a standard construction of a completely positive map using a Hilbert $ C^{*} $-bimodule, and this is part of the KSGNS construction \cite{Lance}. We start with $ A \tens C(X)  $ as an $ A $-$ C(X) $ bimodule and the semi-inner product $ \langle,\rangle:\overline{A \tens C(X)  } \tens_{A} (A \tens C(X)  )\rightarrow C(X)$ defined by
	\begin{equation}\label{innerPr}
	\left\langle \overline{a \tens f}, a' \tens f' \right\rangle =f^{*} \delta (a^{*}a') f'\ .
	\end{equation}
Set $N$ to be the space of zero length vectors, i.e. $ \sum a_{i}\tens f_{i} $ so that 
\begin{equation*}
\big\langle \overline{\sum a_{i}\tens f_{i}}, \sum a_{j}\tens f_{j}\big\rangle =0\ .
\end{equation*}
Now we define $ E= (A \tens C(X))/N$. This has completion a Hilbert $ A$-$C(X)$ $C^{*}$-bimodule and given $ 1\tens 1 \in E $ we have 
\begin{equation*}
	\left\langle  \overline{1\tens1}, a.1\tens 1\right\rangle   =\delta(a)\ .
\end{equation*}
\subsection{The matrix algebra case}
	The pure states on $ M_{n}(\C) $ are parametrised by $\underline{v} \in \mathrm{Row}^n(\C)   $ by
\begin{equation}\label{eqn23}
	\phi_{\underline{v}}(a)=\underline{v} a \underline{v}^{*}\in \C
\end{equation}
where $ \underline{v}  \, \underline{v}^{*} =1$ for normalisation \cite{Choi}. Because scalar multiplication of $ \underline{v} $ by a unit norm complex number leaves the state unaffected the space of pure states is the quotient $ X=\C  \mathbb{P}^{n-1} $ of unit vectors in $ \mathrm{Row}^{n}(\C) $ i.e. $ S^{2n-1}  $ quotiented by the circle group $ U_{1} $. We take the positive map $ \delta: M_{n}(\C) \rightarrow C(\C \mathbb{P}^{n-1})$ defined by $ \delta(a)([\underline{v}])=\phi_{\underline{v}}(a) $ for $ \underline{v} \in S^{2n-1} $. We carry out the KSGNS construction given at the beginning of this section for $ A=M_{n}(\C)   $. We write $ M_{n}(\C) \tens C(\C \mathbb{P}^{n-1}) $ as $ \mathrm{Col}^{n}(\C) \tens C(\C \mathbb{P}^{n-1}, \mathrm{Row}^{n}(\C))$, which are isomorphic as $ \mathrm{Row}^{n}(\C)$ is finite dimensional. For $ c_{i}\tens r_{i} \in \mathrm{Col}^{n}(\C) \tens C(\C \mathbb{P}^{n-1}, \mathrm{Row}^{n}(\C))  $ the inner product  in (\ref{innerPr}) is 
\begin{equation}\label{eqn32}
	\langle\overline{	c_{1}\tens r_{1}}, c_{2}\tens r_{2}\rangle ([\underline{v} ])=\underline{v}  r_{1}([\underline{v} ])^{*} c_{1}^{*} c_{2} r_{2}(\underline{v} )\underline{v} ^{*}\in \C
\end{equation}
for $ \underline{v}  \in S^{2n-1}$ a row vector representing an element $ [\underline{v} ] $ of $ \C  \mathbb{P}^{n-1}$. 
\begin{proposition}\label{prop11}
	The quotient of $\mathrm{Col}^{n}(\C) \tens C(\C \mathbb{P}^{n-1}, \mathrm{Row}^{n}(\C))  $ by the length zero vectors $ N $ is isomorphic to $ \mathrm{Col}^{n}(\C) \tens \Gamma  \tau$ where $  \Gamma  \tau$ is the continuous sections of the tautological bundle $ \tau $.
\end{proposition}
\begin{proof}
	For $ \underline{v}\in S^{2n-1} $ we look at the conditions for $ c_{i}\tens r_{i}$ to be in $ N $, which is $ \sum_{ i j}\langle \overline{c_{i} \tens r_{i}}, c_{i} \tens r_{i}\rangle =0$ using (\ref{eqn32}). Using the projection matrix $ P_{i j}=\bar{v}_{i} v_{j} $ we see that 
	\begin{align*}
		\langle c_{1} \tens r_{1}, c_{2},\tens r_{2}\rangle = 	\langle \overline{c_{1} \tens r_{1}P} , c_{2},\tens r_{2} P\rangle 
	\end{align*}
	just using the fact $ v_{i} \bar{v}_{i}=1 $ (summing over $ i $). Thus the null space $ N $ includes all $ c\tens r (1-P) $ and the only possible non-null elements are $ c \tens r P$, which is $ c \tens s $ where $ s $ is a multiple of $ \underline{v} $. A quick check shows that all these are not null (except $ 0 $). 
\end{proof}
The sections $ \Gamma \tau $ of $ \tau $ are identified with $ C_{-1} (\C  \mathbb{P}^{n-1} ) $ and so we have $  \mathrm{Col}^{n}(\C)  \tens C_{-1}(\C  \mathbb{P}^{n-1} )  $ with inner product
\begin{align}\label{eqn7}
	\langle\overline{c_{1}\tens f_{1}}, c_{2}\tens f_{2}\rangle=c_{1}^{*} c_{2} f_{1}^{*} f_{2}\in C(\C \mathbb{P}^{n-1})
\end{align}
and this a Hilbert $ M_{n} $-$ C(\C \mathbb{P}^{n-1}) $ $ C^{*} $-bimodule. Finally we consider $1\tens 1\in M_{n}(\C)\tens C(\C\mathbb{P}^{n-1})  $ and find $ e_{1\tens 1}=[1\tens1]\in  \mathrm{Col}^{n}(\C)  \tens C_{-1}(\C  \mathbb{P}^{n-1} ) $ under our isomorphism from Proposition \ref{prop11}. Take $ h_{i} $ to be the column vector with $ 1 $ in position $ i $ and zero elsewhere. Then in $ \mathrm{Col}^{n }(\C) \tens C(\C  \mathbb{P}^{n-1} , \mathrm{Row}^{n}(\C))$ $ e_{1\tens 1}=[1\tens 1 ]$ corresponds to $ h_{i} \tens h_{i} ^{*}$ summing over $ i $. Using the isomorphism from Section \ref{subsec1} between $ \Gamma \tau $ and $ C_{-1}(\C \mathbb{P}^{n-1}) $, $ e_{1\tens 1}=[1\tens 1] $ corresponds to $ h_{i} \tens \bar{v}_{i}\in \mathrm{Col}^{n}(\C) \tens C_{-1}(\C  \mathbb{P}^{n-1})$ summing over $ i $. Under the isomorphism we adapt (\ref{eqn23}) to give $ \phi: M_{n} \rightarrow  C(\C \mathbb{P}^{n-1} )$, for $ a=(a_{i j}) \in M_{n} $
\begin{align}\label{eqn31}
	\phi(a)=\sum_{ i j} \langle \overline{ h_{i}\tens \bar{v}_{i} }  \tens a h_{j} \tens \bar{v}_{j}\rangle= \sum_{ i j} v_{i} a_{ij} \bar{v}_{j}\ ,
\end{align}
and this is the state evaluation map. 
\section{Connections on the Hilbert $ C^{*} $-bimodule }\label{Se4}
We now have a formula (\ref{eqn31}) for the state evaluation map using bimodules, and would like to ask whether it is differentiable. To do this we use a bimodule connection. The first thing to do is to take the smooth functions as a subset of our Hilbert $ C^{*} $-bimodule $ \mathrm{Col}^{n}(\C) \tens C_{-1} (\C \mathbb{P}^{n-1})$ by setting $ E=  \mathrm{Col}^{n}(\C) \tens C_{-1}^{\infty} (\C \mathbb{P}^{n-1}) $.
\subsection{ Inner product preserving connections on $ E=\mathrm{Col}^{n}(\C) \tens C^{\infty}_{-1}(\C  \mathbb{P}^{n-1} )  $}\label{Se2}
We have generators of $  C^{\infty}_{-1}(\C  \mathbb{P}^{n-1} ) $, the smooth sections of  $ \tau $, given by $ \bar{v}_{i}$ and a projection matrix $  Q_{i j}=v_{i} \bar{v}_{j}$ so that $ \bar{v}_{i}  Q_{i j}= \bar{v}_{j}$.  We specify a right connection 
\begin{align*}
	\nabla_{E}: \mathrm{Col}^{n}(\C) \tens C_{-1}(\C  \mathbb{P}^{n-1} )  \rightarrow \mathrm{Col}^{n}(\C) \tens C_{-1}(\C  \mathbb{P}^{n-1} ) \tens_{C^{\infty}(\C \mathbb{P}^{n-1}) } \Omega^{1} (\C  \mathbb{P}^{n-1} ) 
\end{align*}
by, for some $\Gamma^{p q}{}_{ij} \in \Omega^{1} (\C  \mathbb{P}^{n-1} ) $ and summing over repeated indices
 \begin{align}\label{eqn37}
 	\nabla_{E} (h_{i}\tens\bar{v}_{j})=h_{p}\tens \bar{v}_{q}\tens \Gamma^{p q}{}_{i j}\ .
 \end{align}
As 
\begin{align*}
	h_{p}\tens \bar{ v}_{q} \tens	\Gamma ^{pq}{}_{ij} &= h_{p}\tens \bar{ v}_{s} Q _{sq}\tens	\Gamma ^{pq}{}_{ij} = h_{p}\tens \bar{ v}_{s} \tens	Q _{sq} \Gamma ^{pq}{}_{ij} 
\end{align*}
we can suppose without loss of generality that 
\begin{align}\label{eqn1}
	\Gamma ^{pq}{}_{ij} = Q _{qs} \Gamma ^{ps}{}_{ij} \ .
\end{align}
Also using $ \bar{v}_{j}=\bar{v}_{q} Q_{q j} $
 \begin{align*}
 	\nabla (h_{i}\tens \bar{ v}_{j} Q_{jk})&=h_{i}\tens \bar{ v}_{j} \tens \di Q_{jk}+h_{p}\tens\bar{ v}_{q} \tens  \Gamma ^{pq}{}_{ij} Q_{jk} \cr
 	&=\nabla (h_{i}\tens \bar{ v}_{k} )= h_{p}\tens\bar{ v}_{q} \tens  \Gamma ^{pq}{}_{ik} \ ,
 \end{align*}
so we have
 \begin{align}\label{eqn2}
 	\Gamma ^{pq}{}_{ij} (\delta_{jk}-Q_{jk})=\delta_{pi} Q_{qj} \di Q_{jk}\ .
 \end{align}
Thus for a right connection (\ref{eqn37}) we require (\ref{eqn1}) and (\ref{eqn2}) to be satisfied. 
\begin{proposition}\label{prop2}
	The connection (\ref{eqn37}) is a bimodule connection with
	\begin{equation*}
		\sigma_{E} : \Omega^{1}_{\mathrm{uni}} (M_{n} (\C)) \tens_{M_{n} (\C)} E \rightarrow E \tens\Omega^{1} (\C\mathbb{P}^{n-1})
	\end{equation*}
	extending to a bimodule map 
	\begin{equation*}
		\hat{\sigma }_{E} : M_{n} (\C) \tens M_{n} (\C)  \tens_{M_{n} (\C)} E \rightarrow E \tens\Omega^{1} (\C\mathbb{P}^{n-1})
	\end{equation*}
	by the formula, for $ E_{ij} $ the standard matrix with $ 1 $ in row  $ i $ column $ j $ and zero elsewhere
	\begin{equation*}
		\hat{\sigma }_{E} (E_{a b} \tens E _{s t} \tens h_{i} \tens \bar{v}_{j})= \delta_{ti} h_{a} \tens \bar{v}_{q} \tens  \Gamma^{b q}{}_{sj} \ .
	\end{equation*}
\end{proposition}
\begin{proof}
The bimodule connection condition gives, 
\begin{align}\label{eqn18}
	\sigma_{E}(\di E_{st} \tens h_{i}\tens \bar{v}_{j})&=\nabla_{E} (E_{st} h_{i}\tens \bar{v}_{j})-E_{st} \nabla_{E} (h_{i}\tens \bar{v}_{j})\cr
	&=\delta_{ti} \nabla _{E}( h_{s}\tens \bar{v}_{j})-E_{st} \nabla _{E}(h_{i}\tens \bar{v}_{j})\cr
	&=\delta_{ti} h_{p}\tens \bar{v}_{q}\tens \Gamma ^{pq}{}_{sj} - E_{st} h_{p}\tens \bar{v}_{q}\tens \Gamma ^{pq}{}_{ij} \cr
	&=\big(\delta_{ti} h_{p} \delta_{sr}-\delta_{t p} h_{s}\delta_{ri}\big)\tens  \bar{v}_{q} \tens \Gamma ^{pq}{}_{rj} \ .
\end{align}
	 Note that $ \hat{\sigma}_{E} $ is explicitly a left module map and is extended to a right $ C (\C \mathbb{P}^{n-1}) $ module map by multiplication on the rightmost factor. Then for the universal calculus we get $ \di E_{s t}= I _{n} \tens E_{st}-E_{st} \tens I_{n} $, and summing over $ k $
	\begin{align*}
		\hat{\sigma }_{E}(\di E_{st} \tens h_{i} \tens \bar{v}_{j} )&= \hat{\sigma}_{E}  (E_{pp} \tens E_{st} \tens h_{i}\tens \bar{v}_{j})-\hat{\sigma}_{E}  (E_{st} \tens E_{pp} \tens h_{i}\tens \bar{v}_{j})\cr 
		&=\delta_{ti} h_{p} \bar{v}_{q}  \tens \Gamma^{pq }{}_{r j}  \delta_{sr}-\delta_{tp} \delta_{ri} h_{s} \tens\bar{v}_{q}  \tens \Gamma^{pq }{}_{r j}  
	\end{align*}
	which agrees with (\ref{eqn18}).
\end{proof}
The curvature of the connection is given by 
\begin{align*}
	R_{E} (h_{i} \tens \bar{v}_{j})&=(\id \tens \di + \nabla_{E}\wedge \id ) \nabla_{E}(h_{i}\tens \bar{v}_{j})\cr
	&= h_{p}\tens \bar{v}_{q} \tens \di\, \Gamma ^{pq}{}_{ij} +h_{s}\tens  \bar{v}_{t} \tens \Gamma ^{st}{}_{pq} \wedge \Gamma ^{pq}{}_{ij}\cr 
		&= h_{p}\tens \bar{v}_{q} \tens \big(\di\, \Gamma ^{pq}{}_{ij} +    \Gamma ^{pq}{}_{st} \wedge \Gamma ^{st}{}_{ij}\big)\ .
\end{align*}
We set $ X^{pq}{}_{ij} =\di \Gamma^{p q}{}_{ij}+    \Gamma ^{pq}{}_{st} \wedge \Gamma ^{st}{}_{ij}$ so
\begin{equation}\label{eqn27}
	R_{E}(h_{i}\tens \bar{v}_{j})=h_{p}\tens \bar{v}_{q} \tens X^{pq}{}_{ij}\ .
\end{equation}
Using (\ref{eqn9}) and where 
$ E_{rt} $ is the matrix with $1  $ in row $ p $ column $t   $ and zero elsewhere 
\begin{align}\label{eqn16}
	R_{E}(E_{rt}h_{i}\tens \bar{v}_{j})-	E_{rt} R_{E}(h_{i}\tens \bar{v}_{j})
	&=\delta_{ti} h_{p}\tens \bar{v}_{q}  \tens X^{pq}{}_{rj}- E_{rt} h_{p}\tens \bar{v}_{q} \tens X^{pq}{}_{ij}\cr 
&=\delta_{ti} h_{p}\tens \bar{v}_{q}  \tens X^{pq}{}_{rj}- \delta_{tp} h_{r}\tens \bar{v}_{q} \tens X^{pq}{}_{ij} \ .  
\end{align}
We see that the curvature is not necessarily a left module map through by general theory it must be a right module map. 

We require two additional properties of our connection, that it preserves the inner product (\ref{eqn7}) and that it vanishes on $ e_{1\tens 1} $. The inner product from (\ref{eqn7}) gives
\begin{align}\label{eqn14}
	\langle \overline{h_{s}\tens \bar{ v}_{t}}, h_{i}\tens \bar{ v}_{j}\rangle=\delta_{si} v_{t} \bar{ v}_{j}
\end{align}
and for the connection (\ref{eqn37}) to preserve the inner product we require
\begin{align}\label{eqn3}
	\delta_{is } \di (v_{t}\bar{ v}_{j})&= 	\langle \overline{h_{s}\tens \bar{ v}_{t}}, h_{p}\tens \bar{ v}_{q}\rangle \Gamma ^{pq}{}_{ij} +\big(\Gamma ^{pq}{}_{st} \big)^{*} \langle \overline{h_{p}\tens \bar{ v}_{q}}, h_{i}\tens \bar{ v}_{j}\rangle \cr
	&= 	\delta_{sp} v_{t}\bar{ v}_{q} \Gamma ^{pq}{}_{ij}+\big(\Gamma ^{pq}{}_{st} \big)^{*} \delta_{pi} v_{q}\bar{ v}_{j}\ .
\end{align}
We also need for $ \nabla_{E} (e_{1\tens 1})=0 $
\begin{align}\label{eqn12}
	0=\nabla_{E}(h_{i}\tens \bar{v}_{i})= h_{p}\tens \bar{v}_{q}\tens \Gamma ^{pq}{}_{ii} 
\end{align}
so $ \Gamma ^{pq}{}_{ii} =0 $.  
\subsection{ A simple example of the connection}\label{sub1}
Here we find a simple example of a connection satisfying the previous conditions in Section \ref{Se2}. From (\ref{eqn1}) we have  $ \Gamma^{pq}{}_{rs} =  v_{q} C^{p}{}_{rs} $ where $ C ^{p}{}_{rs} =\bar{v}_{j}\Gamma^{pj}{}_{rs} $.  Now (\ref{eqn2}) becomes  
\begin{align*}
	v_{q} \,  C^{p}{}_{i j} (\delta_{jk}-Q_{jk})=\delta_{pi} v_{q}\bar{v}_{s}\di  (v_{s}\bar{v}_{k})
\end{align*}
and as this is true for all $ q $ we deduce, using the relations for $ \Omega^{1}(\C \mathbb{P}^{n-1}) $ 
\begin{align}\label{eqn4}
	C^{p}{}_{i j} (\delta_{jk}-Q_{jk})=\delta_{pi} \bar{v}_{s}\big(\di  v_{s}\bar{v}_{k}+v_{s}\di \bar{v}_{k}\big)=\delta_{pi}  \di \bar{v}_{k}\ .
\end{align}
Also (\ref{eqn3}) gives
\begin{align}\label{eqn5}
	\delta_{is} \di Q_{tj}&= \delta_{sp} v_{t} \bar{v}_{q} v_{q} \,C^{p}{}_{ij} + \bar{v}_{q} (C^{p}{}_{st})^{*} \delta_{pi} v_{q}  \bar{v}_{j}   \cr
	&= \delta_{sp} v_{t} \,C^{p}{}_{ij}+ \delta_{pi} \bar{v}_{j}   (C^{p}{}_{st})^{*} = v_{t}\, C^{s}{}_{ij} +\bar{v}_{j}   (C^{i}{}_{st})^{*}
\end{align}
Thus we have for a right connection (\ref{eqn4}), for metric preserving we get (\ref{eqn5}), for $ \nabla(e_{1\tens 1}) =0$ we get $ C^{p}{}_{ii}=0 $. The curvature is 
\begin{align*}
	R_{E}(h_{i}\tens \bar{v}_{j})=h_{p}\tens \bar{v}_{q}\tens \big(\di (v_{q}\, C^{p}{}_{ij})+v_{q}v_{t} C^{p}{}_{st} \wedge C^{s}{}_{ij}\big)
\end{align*}
and using $ \bar{v}_{q}=\bar{v}_{m} v_{m} \bar{v}_{q}$
\begin{align}\label{eqn6}
	R_{E}(h_{i}\tens \bar{v}_{j})&=h_{p}\tens \bar{v}_{m}\tens v_{m}\bar {v}_{q}\big(\di v_{q} \wedge C^{p}{}_{ij}+v_{q} \di C^{p}{}_{ij} +v_{q} v_{t}\, C^{p}{}_{st} \wedge C^{s}{}_{ij}\big)\cr
	&=h_{p}\tens \bar{v}_{m}\tens v_{m}\big(\di C^{p}{}_{ij}+v_{t}\,  C^{p}{}_{st} \wedge  C^{s}{}_{ij} \big)\ .
\end{align}
To simplify this further, from (\ref{eqn4}) we write 
\begin{align*}
	C^{p}{}_{i k} = 	C^{p}{}_{i j} (\delta_{jk}-Q_{jk})+ 	C^{p}{}_{i j}\, Q_{jk}=\delta_{pi} \di \bar{v}_{k}+ 	C^{p}{}_{i j}\, v_{j}\bar{v}_{k} 
\end{align*}
and we set $ D_{pi}=C^{p}{}_{i j} v_{j} $, and then $ 	C^{p}{}_{i k} =\delta_{pi} \di \bar{v}_{k}+ 	D_{pi} \,\bar{v}_{k}  $. Now (\ref{eqn4} ) is automatically true and (\ref{eqn5}) becomes
\begin{align*}
	\delta_{is} (\di v_{t}\bar{v}_{j} + v_{t}\di  \bar{v}_{j} )&=v_{t}(\delta_{si}\di \bar{v}_{j}+D_{s i}\bar{v}_{j})+\bar{v}_{j}(\delta_{is}\di \bar{v}_{t}+D_{is}\bar{v}_{t})^{*}\cr
	&=\delta_{is} (v_{t}\di  \bar{v}_{j} + \bar{v}_{j} \di v_{t} )+v_{t}\bar{v}_{j}(D_{si}+(D_{is})^{*})\ .
\end{align*}
We conclude that for matrix $ D $ we have (\ref{eqn5}) if and only if $ D^{*}+D =0$ as a matrix. Next we require 
\begin{align}\label{eqn8}
	C^{p}{}_{ii}=\delta_{pi} \di \bar{v}_{i}+D_{pi} \bar{v}_{i}= \di \bar{v}_{p}+D_{pi} \bar{v}_{i}=0\ .
\end{align}
Finally we put 
\begin{align*}
	D_{pi}=-\di \bar{v}_{p} v_{i}+\di v_{i}\bar{v}_{p}+G_{p i}\ .
\end{align*}
Now we have from (\ref{eqn8})
\begin{align*}
	D_{pi}\bar{v}_{i}=-\di \bar{v}_{p}+G_{pi}\bar{v}_{i},
\end{align*}
so we have the condition $ G_{pi}\bar{v}_{i}=0 $, and also 
\begin{align*}
	(D_{i p})^{*}=-\di v_{i}\bar{v}_{p}+\di \bar{v}_{p}v_{i}+(G_{i p})^{*}
\end{align*}
so $ D^{*}+D=0 $ if and only if $ G^{*}=-G $. Now we calculate the bracket in the formula for the curvature in (\ref{eqn6}). This is 
\begin{align*}
	\di C^{p}{}_{ij}+C^{p}{}_{st}\,  v_{t}\wedge  C^{s}{}_{ij}=\bar{v}_{j}\big(G_{ps}\wedge G_{si}-v_{i} G_{ps}\wedge \di \bar{v}_{s}+\bar{v}_{p} \di v_{s} \wedge G_{si}+\di G_{pi} -\di v_{i}\wedge \di \bar{v}_{p}\big)\ .
\end{align*}
We can simplify the curvature while satisfying all our conditions simply by putting $ G=0 $, to give 
\begin{align}\label{eqn9}
	R_{E}(h_{i}\tens \bar{v}_{j})=h_{p}\tens \bar{v}_{m} \tens v_{m}\bar{v}_{j} \di \bar{v}_{p} \wedge \di v_{i}=h_{p}\tens \bar{v}_{j} \tens\di \bar{v}_{p} \wedge \di v_{i}\ .
\end{align} 
For completeness we calculate
\begin{align}\label{eqn15}
	\Gamma^{pq}{}_{rs}=v_{q} C^{p}{}_{rs}&=v_{q} \big(     \delta_{pr} \di \bar{v}_{s}+ 	D_{pr} \,\bar{v}_{s}       \big)\cr
	&=v_{q} \big(     \delta_{pr} \di \bar{v}_{s}+ 	\bar{v}_{s}   (-\di \bar{v}_{p} v_{r}+\di v_{r}\bar{v}_{p})    \big)\cr
	&=v_{q} \big(     \delta_{pr} \di \bar{v}_{s}-	\bar{v}_{s}   v_{r} \di \bar{v}_{p} +\bar{v}_{s}\bar{v}_{p} \di v_{r}  \big)\ ,
\end{align}
and from (\ref{eqn27})
\begin{equation}\label{eqn39}
	X^{pq}{}_{ij} =\delta_{qj} \di \bar{v}_{p} \wedge \di v_{i}\ .
\end{equation}
\section{Differentiating positive maps}
We wish to extend the map $ \phi : A \rightarrow B $ defined by $ \phi(a)= \langle \bar{e},a e  \rangle $ in (\ref{eqn31}) to a map of differential forms $ \phi: \Omega_{A}^{m} \rightarrow \Omega_{B}^{m} $. A theory of how to do this is set down in \cite{BMbook}, (using left instead of right connections), but it assumes conditions on the curvature which we do not have and results in a cochain map, so we need to be more careful and give a more general account of the theory, beginning with how $ \sigma_{E} $ extends to a map of differential forms, with general algebras $ A, B $, and bimodule $ W $. 
\subsection{General theory of extendability and curvature}\label{Se3}
We begin with a right handed version of Lemma 3.72 in \cite{BMbook}.  For algebras $ A,B $ with calculi we suppose that $ (\nabla_{W}, \sigma_{W})  $ is a bimodule connection on an $ A $-$ B $ bimodule $ W $.  The curvature $ R_{W} $ of a right bimodule connection must be a right module map, but not necessarily a bimodule map.  
\begin{lemma} \label{Lem1}
	 Given an $ A $-$ B $ bimodule $ W $ with a right bimodule connection $ \nabla_{W}: W \rightarrow W \tens_{B} \Omega_{B}^{1}  $ and $ \sigma_{W}: \Omega_{A}^{1}\tens_{A} W\rightarrow W \tens_{B}\Omega_{B}^{1} $, for the curvature we have 
\begin{gather*}
	R_{W}(a\,e)-a\,R_{W}(e)=(\sigma_{W}\wedge \id)(\di a \tens \nabla_{W}(e))+(\id \tens \di +\nabla _{W}\wedge \id ) \sigma_{W}  (\di a \tens e)\cr
	c R_{W} (a e)-ca R_{W}(e) = 
		(\sigma _{W}\wedge \id ) (c \di a \tens \nabla_{W} e )+(\id \tens \di +\nabla_{W} \wedge \id ) \sigma_{W}(c \di a \tens e )\cr
		\quad- (\sigma_{W} \wedge \id ) (\id \tens \sigma_{W}) (\di c \tens \di a \tens e)\ .
\end{gather*}
\end{lemma}
\begin{proof}
	By definition of $ R_{W} $,
	\begin{align*}
		R_{W}(a e)&= (\id \tens \di +\nabla_{E}\wedge \id ) \nabla_{W} (a e)\cr
		&= (\id \tens \di +\nabla_{W}\wedge \id )\big(\sigma_{W}(\di a \tens  e)+a . \nabla_{W} e\big)\cr
		&=(\id \tens \di +\nabla_{W}\wedge \id )\sigma_{W}(\di a \tens  e)+(\sigma_{W}\wedge \id )(\di a \tens \nabla_{W} e)+a . R_{W}(e)\ .
	\end{align*}
Now multiply the first equation in the statement by $ c\in A $ to get 
\begin{align*}
	c R_{W} (a e )- c a R_{W}(e)= (\sigma \wedge \id ) (c \di a \tens \nabla_{W} e)+c (\id \tens \di +\nabla_{W} \wedge \id ) \sigma_{W} (\di a \tens e )
\end{align*}
and use the definition of $ \sigma_{W} $ again to get the second equation. 
\end{proof}
The  following definition is a right version of extendability from \cite{BMbook}.
\begin{definition}\label{de1}
		 Given an $ A $-$ B $ bimodule $ W $ with a right bimodule connection $ \nabla_{W}: W \rightarrow W \tens_{B} \Omega_{B}^{1}  $ and $ \sigma_{W}:\Omega_{A}^{1}\tens_{A} W \rightarrow W \tens_{B}\Omega_{B}^{1} $, we say\underline{} that $ (\nabla_{W}, \sigma_{W}) $ is extendable if $ \sigma_{W}  $ extends to a map $ \sigma_{W}: \Omega_{A}^{n} \tens_{A} W \rightarrow W\tens_{B} \Omega_{B}^{n} $ such that for all $ \xi , \eta \in \Omega_{A} $
	\begin{align}\label{eqn24}
		\sigma_{W} (\xi \wedge \eta \tens e)=(\sigma_{W} \wedge \id) (\id \tens \sigma)(\xi \tens \eta \tens e)\ .
	\end{align} 
\end{definition}
\begin{cor}\label{cor1}
The $\sigma_{W}$ in Lemma \ref{Lem1} is extendable for the maximal prolongation calculus $ \Omega_{A}^{n} $ if and only if, for all $ c_{i} $, $ a_{i}\in A $ with $ \sum_{ i } c_{i} \di a_{i}=0 \in \Omega^{1}_{A} $
\begin{align}\label{eqn25}
	\sum_{ i }\big( c_{i} a _{i} R_{W}(e)-c_{i} R_{W} (a_{i} e)\big)=0\ .
\end{align}
\end{cor}
\begin{proof}
	To define a map $ \sigma : \Omega_{A}^{2}\tens_{A} W \rightarrow W \tens_{B} \Omega_{B}^{2} $ by (\ref{eqn24}) where $ \xi $, $ \eta \in \Omega_{A}^{1} $, we require the RHS of (\ref{eqn24}) to vanish for all $ \xi \wedge \eta=0 $ (summation implicit). This is easiest if we have as few relations $ \xi \wedge \eta =0 $ as possible, thus we consider the maximal prolongation. In more detail, if we have $ \sum c_{i}\di a_{i}=0 $ in $ \Omega_{A}^{1} $ then $ \sum \di c_{i} \tens \di a_{i} $ is in the kernel of $ \wedge $ and we then have from Lemma \ref{Lem1}
	\begin{align}\label{eqn10}
		\sum \big(c_{i}a_{i} R_{W}(e)-c_{i} R_{W}(a_{i}e)\big)=\sum_{i} (\sigma _{W}\wedge \id)(\id \tens \sigma_{W})(\di c_{i}\tens \di a_{i} \tens e)\ .
	\end{align}
	Thus we need to show that for all $ \sum c_{i}\di a_{i}=0  $ we have the LHS of (\ref{eqn10}) vanishing. 
\end{proof}
	\begin{cor}\label{cor2}
		Either of the following conditions imply the condition (\ref{eqn25}) in Corollary \ref{cor1}:
		\begin{itemize}
			\item [(a)] $ R_{W} $ is a left module map
			\item [(b)] $ \Omega_{A}^{1} $ is the universal calculus.
		\end{itemize}
	\end{cor}
\begin{proof}
	(a) is  obvious from Corollary \ref{cor1}. For (b), by definition of the first order universal calculus we have 
	\begin{align*}
		\sum c_{i} \di a_{i}=c_{i}\tens a_{i}-c_{i}a_{i}\tens 1  \in A\tens A
	\end{align*}
	and if this vanishes then so does the LHS of (\ref{eqn10}). 
\end{proof}
Now we assume extendability for $ \sigma_{W} $ and work out the consequences.
\begin{proposition}\label{prop6}
Given the conditions of Lemma \ref{Lem1} and assuming that $ \sigma_{W} $ is extendable, the map $ S_{W}: \Omega_{A}^{n} \tens_{A} W \rightarrow W \tens_{B} \Omega_{B}^{n+1} $ defined by
	\begin{align}\label{eqn26}
	S_{W}(\xi \tens e)= (\sigma_{W} \wedge \id) (\xi \tens \nabla_{W} e) -(\id \tens\di +\nabla_{W} \wedge\id ) \sigma_{W}(\xi \tens e )(-1)^{|\xi|}+ \sigma_{W} (\di  \xi \tens e)(-1)^{|\xi|}
\end{align}
is a well defined bimodule map, and 
	\begin{align}\label{eqn35}
	S_{W}(\xi \wedge \kappa \tens e)=(\sigma_{W} \wedge \id )(\id \tens S_{W}) (\xi \tens \kappa\tens e)+(-1)^{|\kappa|}(S_{W} \wedge \id )(\id \tens \sigma_{W})(\xi \tens \kappa \tens e)\ .
\end{align}
For the derivative of $ S_{W} $ we have 
\begin{align}\label{eqn34}
	\nabla_{R}^{[|\xi|+1]} S_{W}(\xi \tens e )-S_{W}(\di \xi \tens e)&= -(-1)^{|\xi|} (S_{W}\wedge \id )(\xi \tens \nabla_{W} e )\cr 
	&\quad+(-1)^{|\xi|} \big( (\sigma_{W} \wedge \id) (\id \tens R_{W})-(R_{W}\wedge \id ) \sigma_{W}\big)\ .
\end{align}
\end{proposition}
\begin{proof}
	To check that it is well defined we use 
	\begin{align*}
		S_{W}(\xi a \tens e)-S_{W}(\xi \tens a e)&=- (\sigma_{W} \wedge \id)(\xi \tens \sigma_{W}(\di a \tens e ))\cr
		&\quad +\sigma_{W} \big( (d(\xi a)-(\di \xi) a) \tens e\big)(-1)^{|\xi|}\cr
		&=-\sigma_{W}(\xi \wedge \di a \tens e)+ \sigma_{W}(\xi \wedge \di a \tens e )=0 
	\end{align*}
by Definition \ref{de1}. To check that it is a right module map we use, where $ \sigma_{W}(\xi \tens e )=f\tens \eta $
\begin{align*}
	S(\xi \tens e a)- S(\xi \tens e)a&= (\sigma_{W} \wedge \id ) (\xi \tens e \tens \di a )-(\id \tens \di )(\sigma_{W} (\xi \tens e)a)(-1)^{|\xi|}\cr
	&\quad +\big((\id \tens \di )\sigma_{W} (\xi \tens e)\big) a (-1)^{|\xi|}\cr 
	&=  f \tens \eta \wedge \di a - f \tens \di (\eta a)(-1)^{|\xi|}+f \tens \di \eta . a (-1)^{|\xi|}\ .
\end{align*}
To check that it is a left module map we use  
\begin{align*}
(-1)^{|\xi|}	\big(S(a \xi \tens e)-a S(\xi \tens e)\big)&= -(\nabla_{W} \wedge \id )(a . \sigma_{W}(\xi \tens e))+a (\nabla_{W} \wedge \id ) (\sigma_{W} (\xi \tens e ))\cr
&\quad +\sigma_{W} (\di (a \xi )\tens e )- \sigma_{W}(a. \di \xi\tens e)\cr 
	&=-\nabla_{W} (a f) \wedge \eta +a \nabla_{W} (f) \wedge \eta+\sigma_{W} (\di a \wedge \xi \tens e)\cr 
	&=-\sigma_{W} (\di a \tens f )\wedge \eta+\sigma_{W} (\di a \wedge  \xi \tens e  )=0\ .
\end{align*}
To verify the product rule for $ S_{E} $, consider
\begin{align*}
	S_{W}(\xi \wedge \kappa \tens e)- (\sigma_{W} \wedge \id )(\id \tens S_{W}) (\xi \tens \kappa  \tens e)
\end{align*}
and use the Leibniz rule for $ \di $ and extendability. For the last formula (\ref{eqn34}) we use 
\begin{align}\label{eqn36}
	S_{W} (\di a \tens e)&= R_{W}(a e)-a R_{W} (e)
\end{align}
and standard manipulations. Recall that $ R_{W} $ is not necessarily a left module map, but use of (\ref{eqn35}) shows that (\ref{eqn34}) is well defined on $ \Omega_{A} \tens_{A} W $.
\end{proof}

Now suppose that $ A $ and $ B $ are $ * $-algebras with $ * $-calculi. Given an inner product $ \langle , \rangle : \overline{W}\tens_{A} W\rightarrow B $ which is preserved by $ \nabla_{W} $ we extend $ \phi:A \rightarrow B $ defined by $ \phi (a)= \langle \bar{e} , a e\rangle $ where $ \nabla_{W} \, e=0 $ to $ \phi : \Omega_{A}^{n} \rightarrow \Omega_{B}^{n} $ by 
\begin{align} \label{eqn30}
	 \phi(\xi)= (\langle, \rangle \tens \id ) (\bar{e} \tens \sigma_{W}(\xi \tens e))\ .
\end{align}
Under the more restrictive conditions where $ R_{W} $ is a bimodule map \cite{BMbook} $ \phi $ would be a cochain map. However, more generally we find a correction term.
\begin{proposition}\label{prop3}
Assume the conditions of \ref{Lem1} and that $ \sigma_{W} $  is extendable.	If $ \nabla_{W} e = 0 $ and $ \nabla_{W} $ preserves the inner product then
	\begin{align} \label{eqn44}
		\di \phi (\xi)= \phi (\di \xi)- (-1)^{|\xi|}\, (\langle,\rangle \tens \id ) (\bar{e}\tens S_{W} (\xi\tens e))\ .
	\end{align}
\end{proposition}
\begin{proof}
	Apply (\ref{eqn26}) to the formula obtained by differentiating (\ref{eqn30}).
\end{proof}
In Proposition \ref{prop4} we see that under the condition of Lemma \ref{Lem1} there is a functor $ \tens W $ from $ \mathcal{E}_{A} $ to $ \mathcal{E}_{B} $, using the specified connection on the tensor product. 
We would like to calculate the curvature of this tensor product connection, but as we noted before the curvature of $ W$ is not necessarily a left module map, so we need more generality than in \cite{BMbook}.
\begin{proposition}\label{prop7}
	If $ F \in \mathcal{E}_{A} $ and $ (\nabla_{W}, \sigma_{W}) $ is an extendable right bimodule connection on $ W \in{} _{A}\mathcal{M}_{B} $ then the curvature of the tensor product connection is 
	\begin{equation}\label{eqn19}
		R_{F\tens W}=\id \tens R_{W}+(\id \tens \sigma_{W})(R_{F}\tens \id )+(\id \tens S_{W})(\nabla_{F}\tens \id ))\ .
	\end{equation}
	Note: The first and last terms are not well defined on $ F\tens_{A} W $, only their sum is.
\end{proposition}
\begin{proof}
	Standard manipulation.
\end{proof}
\subsection{Applications to the state map on matrices}
We return to our specific case of matrices, projective space and bimodule $ E $. As we are using the universal calculus for matrices, by Corollary \ref{cor2} we know that $ \sigma_{E} $ from Section \ref{Se2} is extendable. It will be convenient to extend the domain of definition of $ \sigma_{E} $ given in proposition \ref{prop2} from $ \Omega_{\mathrm{uni}}^{1} (M_{n}) $ to $ \Omega_{\mathrm{uni}}^{m-1} (M_{n}) $ etc. 
\begin{proposition}\label{prop5}
	Regarding $ \Omega_{\mathrm{uni}}^{m-1} (M_{n}(\C)) $ as a subset of $ M_{n} (\C) ^{\tens{m}} $ we find the formula 
	\begin{equation*}
		\hat{\sigma}_{E} : M_{n} ^{\tens{m}}  \tens_{M_{n}(\C)} E \rightarrow E \tens  \Omega^{m-1}(\C \mathbb{P}^{n-1})
	\end{equation*}
	which restricts to the extension of 
	\begin{equation*}
		\sigma_{E} : \Omega^{m-1}_{\mathrm{uni}} (M_{n} (\C)) \tens_{M_{n} (\C)} E \rightarrow E \tens\Omega^{m-1} (\C\mathbb{P}^{n-1})
	\end{equation*}
	from Section \ref{Se2} given by 
	\begin{align*}
		\hat{\sigma}_{E}\big(E_{a_{1} b_{1}} \tens E_{a_{2} b_{2}}\tens \dots \tens E_{a_{m} b_{m}} \tens h_{i}\tens \bar{v} _{j}\big)=\delta _{b_{m} i} \,h_{a_{1} } \tens \bar{v}_{q_{1}} \tens \Gamma^{b_{1} q_{1}}{}_{a_{2} q_{2}} \wedge \Gamma^{b_{2} q_{2}}{}_{a_{3} q_{3}} \wedge \dots \wedge \Gamma^{b_{m-1}\, q_{m-1}}{}_{a_{m}\,  j} \ .
	\end{align*}
\end{proposition}
\begin{proof}
	By induction. From Proposition \ref{prop2} the formula works for $ m=2 $. Assume that it works for $ m $, and then for $ m+1 $, given $ \xi \in \Omega^{1}_{\mathrm{uni}} (M_{n}(\C))$ and $ \eta = E_{a_{i} b_{1} } \tens  E_{a_{2} b_{2} } \tens \dots  \tens E_{a_{m} b_{m} } \in \Omega^{m-1}_{\mathrm{uni}} (M_{n}(\C))$
	\begin{align*}
		\sigma _{E}(\xi  \wedge \eta \tens h_{i} \tens \bar{v}_{j})&=(\id \tens \wedge ) \sigma_{E}(\xi  \wedge \eta \tens h_{i} \tens \bar{v}_{j})\cr
		&=(\sigma_{E} \wedge \id ) (\xi \tens\sigma_{E}( \eta \tens h_{i} \tens \bar{v}_{j}))\cr
		&= \delta_{b_{m} i} \, \sigma_{E} (\xi \tens h_{a_{1}} \tens \bar{v}_{q_{1}}) \wedge \Gamma^{b_{1} q_{1}}{}_{a_{2} q_{2}} \wedge \Gamma^{b_{2} q_{2}}{}_{a_{3} q_{3}} \wedge \dots \wedge \Gamma^{b_{m-1}\, q_{m-1}}{}_{a_{m}\,  j} \ .
	\end{align*}
	Now put $ \xi= E_{a b} \tens E_{st} $ to get 
	\begin{align*}
		\sigma_{E} (\xi  \wedge \eta \tens h_{i} \tens \bar{v}_{j})
		&= \delta_{b_{m} i} \, 	\hat{\sigma} _{E}(E_{a b} \tens E_{st} \tens h_{a_{1}} \tens \bar{v}_{q_{1}}) \wedge \Gamma^{b_{1} q_{1}}{}_{a_{2} q_{2}} \wedge \Gamma^{b_{2} q_{2}}{}_{a_{3} q_{3}} \wedge \dots \wedge \Gamma^{b_{m-1}\, q_{m-1}}{}_{a_{m}\,  j} \cr
		&= \delta_{b_{m} i} \, \delta_{ta_{1}} h_{a} \tens \bar{v}_{q_{0}}\tens \Gamma^{bq_{0}}{}_{s q_{1}} \wedge \Gamma^{b_{1} q_{1}}{}_{a_{2} q_{2}} \wedge \dots \wedge \Gamma^{b_{m-1}\, q_{m-1}}{}_{a_{m}\,  j} 
	\end{align*}
	and this is exactly what the formula gives on applying $ \hat{\sigma}_{E} $ to $ \xi \wedge \eta \tens h_{i} \tens\bar{v}_{j} $  given 
	\begin{align*}
		E_{ab} \tens E_{st} \tens_{M_{n} (\C)} \xi=  \delta_{ta_{1}}E_{ab} \tens E_{s b_{1} } \tens E_{a_{2} b_{2}} \tens \dots \tens E_{a_{m} b_{m}}\ .
	\end{align*}
\end{proof}
We can now extend the state evaluation map $ \phi: M_{n}(\C) \rightarrow C(\C \mathbb{P}^{n-1}) $ from (\ref{eqn23}) and (\ref{eqn31}) to forms by using (\ref{eqn30}).
\begin{cor}\label{cor4}
	The function $ \phi: \Omega_{\mathrm{uni}}^{m-1} (M_{n}(\C))  \rightarrow \Omega^{m-1}(\C \mathbb{P}^{n-1})$ is given by 
	\begin{align*}
		\phi(E_{a_{1} b_{1} }\tens\dots \tens  E_{a_{m} b_{m}})=v_{a_{1}} \bar{v}_{q_{1}}  \Gamma^{b_{1} q_{1}}{}_{a_{2} q_{2}} \wedge \dots \wedge  \Gamma^{b_{m-1} q_{m-1}}{}_{a_{m} b_{m}} 
	\end{align*}
	summing over $ q_{1}, \dots q_{m-1} $.
\end{cor}
\begin{proof}
		Summing over $ i,j $
	\begin{align*}
		\phi(E_{{a_{1} b_{1}} } \tens\dots \tens E_{a_{m} b_{m}})&= (\langle ,\rangle \tens \id )(\id \tens \hat{\sigma }_{E})\big(\overline{h_{j} \tens \bar{v}_{j}} \tens E_{a_{1} b_{1}} \tens \dots \tens E_{a_{m} b_{m}} \tens h_{i} \tens \bar{v}_{i}\big)\cr
		&=  \delta_{b_{m} i} \langle \overline{h_{j} \tens \bar{v}_{j}} , h_{a_{1}} \tens \bar{v}_{q_{1}}\rangle  \Gamma^{b_{1} q_{1}}{}_{a_{2} q_{2}} \wedge  \Gamma^{b_{2} q_{2}}{}_{a_{3} q_{3}} \wedge \dots \wedge  \Gamma^{b_{m-1} q_{m-1}}{}_{a_{m} i}    \cr
		&=  \delta_{b_{m} i}  \delta_{ja _{1}}  v_{j} \bar{v}_{q_{1}} \Gamma^{b_{1} q_{1}}{}_{a_{2} q_{2}} \wedge  \Gamma^{b_{2} q_{2}}{}_{a_{3} q_{3}} \wedge \dots \wedge  \Gamma^{b_{m-1} q_{m-1}}{}_{a_{m} i} \ .
	\end{align*}
\end{proof}
\begin{proposition}\label{prop9}
	Similarly to $ \hat{\sigma}_{E}  $, we can calculate an extension $ \hat{S}_{E} $ of $ S_{E} $ to $ M_{n}\tens M_{n} $ instead of just $ \Omega_{\mathrm{uni}}^{1} (M_{n}) $, giving
	\begin{align*}
		\hat{S}_{E}(E_{a b} \tens E_{rt} \tens h_{i} \tens \bar{v}_{j})&=\delta_{ti} h_{a} \tens \bar{v}_{q} \tens X^{bq}{}_{r j}
	\end{align*}
and this extends to higher forms by
		\begin{align*}
		&\hat{S}_{E}\big(E_{a_{1} b_{1}} \tens E_{a_{2} b_{2}}\tens \dots \tens E_{a_{m} b_{m}} \tens h_{i}\tens \bar{v} _{j}\big)=\delta _{b_{m} i} \,h_{a_{1} } \tens \bar{v}_{q_{1}} \tens\cr
		&\left( \begin{array}{c}
		\Gamma^{b_{1} q_{1}}{}_{a_{2} q_{2}} \wedge \Gamma^{b_{2} q_{2}}{}_{a_{3} q_{3}} \wedge \dots \wedge X^{b_{m-1}\, q_{m-1}}{}_{a_{m}\,  j}+\dots  \\ +(-1)^{m-3} \Gamma^{b_{1} q_{1}}{}_{a_{2} q_{2}} \wedge X^{b_{2} q_{2}}{}_{a_{3} q_{3}} \wedge \dots \wedge \Gamma^{b_{m-1}\, q_{m-1}}{}_{a_{m}\,  j} 	\\+ (-1)^{m-2} X^{b_{1} q_{1}}{}_{a_{2} q_{2}} \wedge \Gamma^{b_{2} q_{2}}{}_{a_{3} q_{3}} \wedge \dots \wedge \Gamma^{b_{m-1}\, q_{m-1}}{}_{a_{m}\,  j} 			
		\end{array}
		\right) 
	\end{align*}
where the wedge products alternate in sign and contain exactly one $ X $ factor.
\end{proposition}
\begin{proof}
	We use (\ref{eqn16}) and (\ref{eqn36}) to find the first equation, using 
		\begin{align}\label{eqn28}
		E_{ab} S_{E} (\di E_{rt} \tens h_{i} \tens \bar{v}_{j})=\hat{S} _{E } \big( (E_{a b} \tens E_{rt} -E_{at} \tens 1\delta_{b r}) \tens h_{i} \tens \bar{v}_{j}\big)\ .  
	\end{align}
The rest is a proof by induction similarly to Proposition \ref{prop5} using Proposition \ref{prop6}.
\end{proof}
\section{Matrix modules and sheaves on $ \C \mathbb{P}^{n-1} $}
\subsection{Differentiating the state evaluation map}
We would like the state evaluation map extended to forms in Corollary \ref{cor4} to be a cochain map, i.e. $ \di \phi(\xi)=\phi (\di \xi) $. However, Proposition \ref{prop3} gives an additional term which we must evaluate. 
\begin{proposition}
	For the usual calculus on projective space the state evaluation map (\ref{eqn31}) is not a cochain map to the standard $ \di $ calculus on $ \C \mathbb{P}^{n-1} $.
\end{proposition}
\begin{proof}
	Using Proposition \ref{prop9} and (\ref{eqn39}) we evaluate the last term in (\ref{eqn44})
	\begin{align*}
		(\langle , \rangle \tens \id )\big( \overline{e_{1\tens 1} }  \tens S_{E} (E_{ab} \tens E_{rt} \tens e_{1\tens 1})\big)&= (\langle , \rangle \tens \id ) \big(\overline{ h_{k}\tens \bar{v}_{k} } \, S_{E}\tens (E_{ab}\tens E_{rt} \tens h_{i}\tens \bar{v}_{i})\big)\cr
		&= \langle \overline{ h_{k}\tens \bar{v}_{k} } , h_{a}\tens \bar{v}_{q}\rangle \delta_{ti} \,X^{bq}{}_{ri}\cr
		&= v_{a} \bar{v}_{q} \delta_{q t} \di \bar{v}_{b} \wedge \di v_{r} = v_{a} \bar{v}_{t} \di \bar{v}_{b} \wedge \di v_{r}\ ,
	\end{align*}
which is nonzero. Now if $ b \neq r $ then $ E_{ab} \tens E_{rt} \in \Omega^{1}_{\mathrm{uni}} (M_{n})$.
\end{proof}
This may seen disappointing, but it is an opportunity to consider the holomorphic structure or projective space. From Definition \ref{de3} and using  (\ref{eqn9}) we see that $ E=\mathrm{Col}^{n} (\C) \tens C_{-1}(\C \mathbb{P}^{n-1}) $ with the connection in Section \ref{Se2} is a holomorphic bundle over $ \C \mathbb{P}^{n-1} $.
\begin{theorem}\label{thm2}
		For the $ \bar{\partial} $ calculus on $ \C \mathbb{P}^{n-1} $ and  the universal calculus on $ M_{n} $ the state evaluation map (\ref{eqn41}) and its extension to forms in Corollary \ref{cor4} is a cochain map.
\end{theorem}
\begin{proof}
	Proposition \ref{prop3} will give the result if the $ S_{E} $ then gives zero in the $ \bar{\partial} $ calculus. This can be seen from Proposition \ref{prop9} and (\ref{eqn39}).
\end{proof}
Using the $ \bar{\partial} $ calculus on $ \C \mathbb{P}^{n-1} $ raises the possibility that the bimodule $E=\mathrm{Col}^{n} (\C) \tens C_{-1}(\C \mathbb{P}^{n-1}) $ could be use to give a functor from $ M_{n} $ modules on $ \C \mathbb{P}^{n-1} $. First we need to consider $ M_{n} $ modules with connection.
\subsection{Connections on right modules over $ M_{n}(\C) $}\label{subsec4}
In this subsection and the next we take $ r_{i}\in \mathrm{Row}^{n}(\C) $ to be the row vector with $ 1 $ in position $ i $ and zero elsewhere. 
\begin{proposition} \label{prop8}
Take the right $ M_{n}(\C) $ module $ F=V \tens \mathrm{Row}^{n}(\C) $ for  a vector space $ V $, with action given by matrix multiplication
\begin{align*}
	(v\tens r_{i}) \triangleleft E_{j k}= v\tens r_{k} \delta_{i j}\ .
\end{align*}
Then a general right connection $ \nabla_{F} $ for the universal calculus on $ M_{n} $ is  
\begin{align*}
	\nabla_{F}(v\tens r_{i}) \in V\tens \mathrm{Row}^{n}(\C)\tens_{M_{n} } \Omega_{\mathrm{uni}}^{1}(M_{n}) \subset  V\tens \mathrm{Row}^{n}(\C) \tens_{M_{n}} M_{n} \tens M_{n} 
\end{align*}
and using the fact that every $ 1 $-form on $ M_{n}  $ can be written as a sum of $ E_{sj}. \di E_{p i} $  we can write
\begin{equation}\label{eqn42}
	\nabla_{F}(v  \tens r_{i})= \sum_{ pj} L_{j p}(v)\tens r_{j} \tens \di E_{p i}
\end{equation}
for linear $ L_{j p}: V \rightarrow V $ with $ \sum _{j} L_{j j}(v)=v $. 
 The curvature of the connection is 
\begin{align*}
	R_{F}(v\tens r_{i})=\sum_{a b j p} L_{a b}\big(L_{j p}(v)\big) \tens r_{a}\tens \di E_{b j}\wedge \di E _{p i}\ .
\end{align*}
\end{proposition}
\begin{proof}
By using the $\mathrm{Row}^{n}(\C) \tens_{M_{n} } M_{n} \cong   \mathrm{Row}^{n}  (\C)$ we get 
\begin{align*}
		(V \tens \mathrm{Row}^{n}(\C) \tens_{M_{n} } \Omega_{\mathrm{uni}}^{1}(M_{n})\cong V\tens K 
\end{align*}
where $ K=\ker \cdot :  \mathrm{Row}^{n} (\C) \tens M_{n} \rightarrow \mathrm{Row}^{n} (\C)$. We write summing over $ j , p, q $
\begin{align*}
	\nabla_{F}(v\tens r_{i}) =S_{i j p q} (v)\tens r_{j} \tens E_{p q}
\end{align*}
and for this to be in $ V\tens K $ we need $ S_{i j p q} (v)\tens \delta_{j p} r_{q}=0 $ i.e. $ \sum_{j}  S_{i j j q}=0 $ for all $ i,q $. We will also write
\begin{align*}
	\nabla _{F} (v\tens r_{i}) = S_{i j p q}(v) \tens r_{j} \tens \di E _{p q} \ .
\end{align*}
and these are the same under the isomorphism as 
\begin{align*}
	 S_{i j p q}(v)\tens r_{j}(I_{n}\tens E_{p q}-E_{p q}\tens I)=S_{i j p q}(v)\tens r_{j}\tens  E_{p q}- S_{i j j q}(v)\tens r_{q}\tens I\ .
\end{align*}
The condition to be a right connection is, for all $ i, s, t $
\begin{align*}
	\nabla_{F}(v\tens r_{i} E_{st})= \nabla (v\tens r_{i}) E_{s t}+v \tens r_{i} \tens \di E_{s t}
\end{align*}
which gives, summing over $ j, p,q $ 
\begin{align*}
	\delta_{ i s} S_{t j p q} (v)\tens r_{j}\tens E_{p q}= S_{i j p q}(v) \tens r_{j} \tens E_{p q} E_{s t}+v \tens r_{i} \tens E_{s t}- \delta_{ i s} v \tens r_{t}\tens I\ .
\end{align*}
This has general solution 
\begin{align*}
	S_{i j p q} (v)=-v \delta_{i j} \delta_{p q}+\delta_{i q} L_{j p} (v) 
\end{align*}
where $ \sum_{j} L_{j j}(v)=v$. 
\end{proof}
If we take ${} _{M_{n}}\mathcal{M}  $ to be the category of  left $ M_{n} $ modules and module maps, then there is a functor ${} _{M_{n}}\mathcal{M} \rightarrow \mathcal{E}_{M_{n}}$ to the category of right $ M_{n} $ modules with right connections for the universal calculus. This is given by $ V \mapsto V \tens \mathrm{Row}^{n}(\C) $ and this is given the connection in Proposition \ref{prop8}, where we define $ L_{i j}(v) $ by the right action $ E_{i j}\triangleright v=L_{ij}(v) $. The condition $ \sum_{ j} L_{jj}(v)=v $ is simply $ I_{n} \triangleright v =v$. Note that this will not give the most general $ L_{i j} $ for Proposition \ref{prop8}, but the restriction to certain $ L_{i j} $ is what we need in the next part.
\subsection{Induced Holomorphic bundles on $ \C \mathbb{P}^{n-1} $}
From Proposition \ref{prop4} we know that there is a functor $ \tens E $ from $ \mathcal{E}_{M_{n}} $ to $ \mathcal{E}_{C(\C \mathbb{P}^{n-1})} $. At the end of the last section we had a functor from ${} _{M_{n}}\mathcal{M}  $ to $ \mathcal{E}_{M_{n}} $, and of course these can be composed. However we know that the state evaluation map $ \phi $ is not a cochain map for the ordinary calculus on $ \C \mathbb{P}^{n-1} $ (using the choice of connection in Section \ref{sub1}), but it is for the $ \bar{\partial} $ calculus. It is then natural to ask if we get a functor into holomorphic bundles on $ \C \mathbb{P}^{n-1} $. We use $ \pi^{i,j} $ for the projection from $ \Omega^{i+j} $ to $ \Omega^{i,j} $. 
 
Given a connection for the calculus $ \Omega ^{n}(\C\mathbb{P}^{n-1}) $ we can get a $ \bar{\partial}  $ connection (see Section \ref{subsec2} ) simply by composing with $ \pi^{0,1}$. Then to have $ F \tens_{M_{n}} E $ being a homomorphic bimodule we require that the $ \Omega^{0,2} $ part of its curvature $ R_{F \tens E} $ vanishes.
\begin{proposition}\label{prop10}
The $ \Omega^{0,2} $ component of the curvature of $ F\tens_{M_{n}} E $ is
\begin{align*}
	&(\id \tens \pi^{0,2}) R_{F\tens E} (v \tens r_{t} \tens h_{i}\tens v_{j})\cr
	&=\sum_{ a b s }L_{ca}L_{bs}(v)\tens \bar{v}_{j} \tens \delta_{ti}   v_{a} v_{s} \tens \di \bar{v}_{g}\wedge \di \bar{v}_{b}	-\sum_{ a }	L_{ca} (v)\tens \bar{v}_{j} \tens v_{a} v_{i} \tens \di \bar{v}_{g}\wedge \di \bar{v}_{t}	\cr
	&\quad + \sum_{ a b}	L_{cg }L_{ba}(v) \tens \bar{v}_{j} \tens v_{a} v_{i}\tens \di \bar{v}_{b}\wedge \di \bar{v}_{t}	
\end{align*}
and in particular if $ L_{c g} L_{e s }= \delta_{ g e } L_{c s}$ then $ (\id \tens \pi^{0,2}) R_{F\tens E}=0 $. 
\end{proposition}
\begin{proof}
 From Proposition \ref{prop7} $ R_{F\tens E} $ splits into three bits, and the $ \id \tens R_{E} $ term does not have a $ \di \bar{v}_{i}\wedge \di \bar{v}_{j} $ part as computed in (\ref{eqn9}). By Proposition \ref{prop9} and equation (\ref{eqn39}) the last term in the formula for $ R_{F\tens E} $ in Proposition \ref{prop7} does not have a $ \Omega^{0,2} $ part either, so we are left with 
 \begin{equation*}
 		(\id \tens \pi^{0,2}) R_{F\tens E}=	(\id \tens \pi^{0,2})  \big((\id \tens \sigma_{E} ) (R_{F} \tens \id)\big)\ .
 \end{equation*}
    Using (\ref{eqn18}) twice we get
\begin{align}\label{eqn20}
	&(\id \tens 	\pi^{0,2})\sigma \big(\di E_{a b} \wedge \di E_{s t} \tens h_{i} \tens \bar{v}_{j}\big)\cr
	&=\sum_{p q g e} \big(\delta_{ti}  \delta_{sr} \delta_{bp} \delta_{a e} h_{g}-\delta_{ti}  \delta_{sr} \delta_{bg} \delta_{ep}h_{a}-\delta_{tp}  \delta_{ri} \delta_{bs} \delta_{a e} h_{g}+\delta_{tp}  \delta_{ri} \delta_{bg} \delta_{es}h_{a}\big)\tens \bar{v}_{f}\tens  	\pi^{0,2} (\Gamma^{gf}{}_{eq} \wedge \Gamma^{pq}{}_{rj} )\cr
		&=\sum_{p  g e f r} \big(\delta_{ti}  \delta_{sr} \delta_{bp} \delta_{a e} h_{g}-\delta_{ti}  \delta_{sr} \delta_{bg} \delta_{ep}h_{a}-\delta_{tp}  \delta_{ri} \delta_{bs} \delta_{a e} h_{g}+\delta_{tp}  \delta_{ri} \delta_{bg} \delta_{es}h_{a}\big)\cr
		&\tens \bar{v}_{f}\tens  v_{f} v_{e} \di \bar{v}_{g}\wedge (-\delta_{p r} \di \bar{v}_{j}+ \bar{v}_{j} v_{r} \di \bar{v}_{p})\cr
			&=\sum_{g} h_{g}\tens \bar{v}_{j}\tens \big( \delta_{ti} v_{a} v_{s} \di \bar{v}_{g}\wedge \di \bar{v}_{b} -\delta_{bs} v_{a} v_{i} \di \bar{v}_{g}\wedge \di \bar{v}_{t} + \delta_{ag} v_{s} v_{i} \di \bar{v}_{b}\wedge \di \bar{v}_{t}\big)\cr 
				&=\sum_{ge f} h_{g}\tens \bar{v}_{j}\tens \big( \delta_{ti} \delta_{eg} \delta_{bf} v_{a} v_{s}  -\delta_{bs} \delta_{eg} \delta_{tf} v_{a} v_{i}  + \delta_{ag} \delta_{be} \delta_{tf} v_{s} v_{i} \big) \di \bar{v}_{e}\wedge \di \bar{v}_{f}
\end{align}
taking only the $ \di \bar{v} \wedge \di \bar{v} $ component. 

We are left with, using (\ref{eqn20})  
\begin{align}\label{eqn21}
	(\id \tens \pi^{0,2}) R_{F\tens E}&=	(\id \tens \pi^{0,2})  \big((\id \tens \sigma_{E} ) (R_{F} \tens \id)\big) \cr
	&=	(\id \tens \pi^{0,2})  (\id \tens \sigma_{E} )\big(R_{F} (v\tens r_{t}) \tens h_{i} \tens \bar{v}_{j}\big)\cr 
	&=(\id \tens \pi^{0,2}) \big(L_{ca } L_{bs} (v)\tens r_{c} \tens \sigma_{E} (\di E_{a b} \wedge \di E_{st} \tens h_{i} \tens \bar{v}_{j})\big) \cr 
		&=\sum_{g e f a b s c } \big(L_{ca } L_{bs} (v)\tens r_{c} \tens h_{g}\tens \bar{v}_{j}\tens \cr 
		& \big( \delta_{ti} \delta_{eg} \delta_{bf} v_{a} v_{s}  -\delta_{bs} \delta_{eg} \delta_{tf} v_{a} v_{i}  + \delta_{ag} \delta_{be} \delta_{tf} v_{s} v_{i} \big) \di \bar{v}_{e}\wedge \di \bar{v}_{f} 
\end{align}
and for this to vanish we need for all $ t, i,j,g,c $
\begin{align}\label{eqn22}
&	\sum_{e f a b s } \big(L_{ca } L_{bs} (v)\tens \bar{v}_{j}\tens \big( \delta_{ti} \delta_{eg} \delta_{bf} v_{a} v_{s}  -\delta_{bs} \delta_{eg} \delta_{tf} v_{a} v_{i}  + \delta_{ag} \delta_{be} \delta_{tf} v_{s} v_{i} \big) \di \bar{v}_{e}\wedge \di \bar{v}_{f} =0\cr
&=\sum_{ a b s }L_{ca}L_{bs}(v)\tens \bar{v}_{j} \tens \delta_{ti}   v_{a} v_{s} \tens \di \bar{v}_{g}\wedge \di \bar{v}_{b}	-\sum_{ a }	L_{ca} (v)\tens \bar{v}_{j} \tens v_{a} v_{i} \tens \di \bar{v}_{g}\wedge \di \bar{v}_{t}	\cr
&\quad + \sum_{ a b}	L_{cg }L_{ba}(v) \tens \bar{v}_{j} \tens v_{a} v_{i}\tens \di \bar{v}_{b}\wedge \di \bar{v}_{t}	\ .
\end{align}
	If $ L_{c g} L_{e s }= \delta_{ g e } L_{c s}$ then the result of (\ref{eqn21}) is 
	\begin{align*}
			&\sum_{ e f a b s }  L_{cs} (v)\tens \delta_{a b}\big( \delta_{ti} \delta_{eg} \delta_{bf} v_{a} v_{s}  -\delta_{bs} \delta_{eg} \delta_{tf} v_{a} v_{i}  + \delta_{ag} \delta_{be} \delta_{tf} v_{s} v_{i} \big) \di \bar{v}_{e}\wedge \di \bar{v}_{f}\cr 
				&=\sum_{ e f a s }  L_{cs} (v)\tens \big( \delta_{ti} \delta_{eg} \delta_{af} v_{a} v_{s}  -\delta_{as} \delta_{eg} \delta_{tf} v_{a} v_{i}  + \delta_{ag} \delta_{ae} \delta_{tf} v_{s} v_{i} \big) \di \bar{v}_{e}\wedge \di \bar{v}_{f}\cr 
				&=\sum_{ f a s }  L_{cs} (v)\tens v_{s}  \big( \delta_{ti} \delta_{af} v_{a}   -\delta_{as}  \delta_{tf}  v_{i}  + \delta_{ag}  \delta_{tf}  v_{i} \big) \di \bar{v}_{g}\wedge \di \bar{v}_{f}\cr 
					&=\delta_{ti}\sum_{ f s }  L_{cs} (v)\tens v_{s} v_{f} \di \bar{v}_{g}\wedge \di \bar{v}_{f}=0\ .
	\end{align*}
\end{proof}
Note that the conditions $ \sum_{ i } L_{ii} (v)=v $ and that in Proposition \ref{prop10} correspond to $ L_{i j} $ being the left action of the matrix unit $E_{i j}  $ in a representation of $ M_{n} (\C) $.  Set $ F=V \tens \mathrm{Row}^{n}(\C) $ as in Proposition \ref{prop8} then 
\begin{align*}
	F\tens_{M_{n} (\C)} E=V\tens  \mathrm{Row}^{n}(\C)  \tens_{M_{n} (\C)} \mathrm{Col}^{n}(\C) \tens C_{-1}  (\C\mathbb{P}^{n-1})\ .
\end{align*}
For $ w \in V $, using (\ref{eqn37}), Proposition (\ref{prop8}) and (\ref{eqn18})
\begin{align*}\label{eqn40}
		\nabla_{F \tens E} (w\tens r_{a} \tens h_{i} \tens \bar{v}_{j})&=(\id \tens \sigma_{E})\big(\nabla_{F} (w\tens r_{a})\tens (h_{i} \tens \bar{v}_{j})+w \tens r_{a} \tens \nabla_{E} (h_{i} \tens \bar{v}_{j}) \big)\cr 
	&= L_{ps}(w) \tens r_{p}\tens \sigma_{E} (\di E_{s a}\tens h_{i} \tens \bar{v}_{j})+w \tens r_{a} \tens   h_{p} \tens \bar{v}_{q} \tens  \Gamma^{pq}{}_{ij}\cr 
	&= L_{ps}(w) \tens r_{p}\tens (\delta_{ ai }h_{t} \delta_{ sr} -\delta_{ at} h_{s} \delta_{ ri} )\tens \bar{v}_{q} \tens  \Gamma^{pq}{}_{ij}+w \tens r_{a} \tens h_{p}\tens \bar{v}_{q} \tens  \Gamma^{pq}{}_{ij}\cr 
	&=\delta_{ ai } L_{pr} (w)\tens r_{p} \tens h_{t} \bar{v}_{q} \tens  \Gamma^{tq}{}_{rj} +w \tens r_{a} \tens h_{p} \tens \bar{v}_{q} \tens  \Gamma^{pq}{}_{ij}\cr
	&\quad - L_{ps}(w)\tens r_{p} \tens h_{s} \tens \bar{v}_{q} \tens  \Gamma^{aq}{}_{ij} \cr 
	&=\big(\delta_{ ai } L_{pr}(w)\tens r_{p } \tens h_{t}+w \tens r_{a}\tens h_{t} \delta_{ ri }-\delta_{ ri} \delta_{ ta}L_{ps}(w) \tens r_{p}\tens h_{s}\big) \tens \bar{v}_{q} \tens  \Gamma^{tq}{}_{rj} 
\end{align*}
Note $ \mathrm{Row}^{n} (\C) \tens_{M_{n}} \mathrm{Col}^{n}(\C)\cong \C$ by $ r_{a} \tens h_{i} \mapsto \delta_{ ai } \in \C$. Look at the last two terms of the last line of (\ref{eqn40}) using this isomorphism 
\begin{align*}
	\big(w\, \delta_{ at} \delta_{ ri}-\delta_{ ri} \delta_{ ta} L_{ps}(w)\delta_{ps}\big) \tens \bar{v}_{q} \tens  \Gamma^{tq}{}_{rj} =\big(w\, \delta_{ at} - \delta_{ ta} \delta_{ ps}  L_{ps}(w)\big) \tens \bar{v}_{q} \tens  \Gamma^{tq}{}_{ij} =0
\end{align*}
by Proposition (\ref{prop8}). Thus we can use the isomorphism to give a connection on $ F \tens_{M_{n}} E  \cong V \tens C_{-1}(\C \mathbb{P}^{n-1}) $ given by 
\begin{align*}
	\nabla (w \tens \bar{v}_{j})= L_{pr} (w)\tens \bar{v}_{q} \tens  \Gamma^{pq}{}_{rj} \ .
\end{align*}
\begin{cor} \label{cor3}
	For the special case of the connection in (\ref{eqn15}) we find 
	\begin{align}
	\nabla (w \tens \bar{v}_{j})&= L_{pr} (w)\tens \bar{v}_{q} \tens v_{q} \big(\delta_{ pr} \di\bar{v}_{j}  -\bar{v}_{j} v_{r} \di \bar{v}_{p} +\bar{v}_{j} \bar{v}_{p} \di v_{r}\big)\cr
&= w \tens \bar{v}_{q} \tens v_{q}	\di \bar{v}_{j}+ L_{pr} (w)\tens \bar{v}_{q} \tens v_{q} \bar{v}_{j} (\bar{v}_{p} \di \bar{v}_{r}- v_{r} \di \bar{v}_{p})\cr 
&= w \tens \bar{v}_{q} \tens v_{q}	\di \bar{v}_{j}+ L_{pr} (w)\tens \bar{v}_{j} (\bar{v}_{p} \di \bar{v}_{r}- v_{r} \di \bar{v}_{p})
	\end{align}
and this splits into a $ \partial $ and a $ \bar{\partial }$ connection
\begin{align}\label{eqn43}
		\partial_{V} (w \tens \bar{v}_{j})&= L_{pr} (w)\tens \bar{v}_{j} \tens \bar{v}_{p} \di v_{r}\cr
	\bar{\partial}_{V}(w \tens \bar{v}_{j}) &= w \tens \bar{v}_{q}  \tens  v_{q} \di \bar{v}_{j}-L_{pr} (w)\tens \bar{v}_{j} \tens v_{r} \di\bar {v}_{p}\ .
\end{align}
\end{cor}
\begin{proposition}\label{prop12}
	The composition of the given functor $ \tens E : \mathcal{E}_{M_{n}} \rightarrow \mathcal{E}_{C(\C \mathbb{P}^{n-1})} $ and the functor in Section \ref{subsec4} ${} _{M_{n}}\mathcal{M} \rightarrow \mathcal{E}_{M_{n}}$ gives a functor from $ {} _{M_{n}}\mathcal{M}  $ to holomorphic bundles on $ \C \mathbb{P}^{n-1} $. It is given by $ V $ mapping to $ V \tens C_{-1} (\C \mathbb{P}^{n-1}) $ with the $ \bar{\partial}_{V}$ connections given in  Corollary \ref{cor3}.
\end{proposition}
\begin{proof}
The category of holomorphic bundles is given morphisms being module maps commutating with $ \bar{\partial} $ operators as in Subsection \ref{subsec3}. Most of this has been proved in the discussion previously. We explicitly check that we have a functor, i.e. that a $ M_{n} $ module map $ \theta: V \rightarrow Y $ gives a commutating diagram 
\begin{align*} 
	\UseComputerModernTips
	\xymatrix{
		V \tens C_{-1} (\C \mathbb{P}^{n-1}) 	 \ar[d]_{\theta \tens \id}  \ar[r]^{\bar{\partial}_{V}\qquad \qquad \qquad} & V \tens C_{-1} (\C \mathbb{P}^{n-1}) \tens_{C(\C \mathbb{P}^{n-1})} \Omega^{0,1} (\C \mathbb{P}^{n-1})\ar[d]_{\theta \tens \id \tens \id}
		\\
		Y\tens  C_{-1} (\C \mathbb{P}^{n-1})   \ar[r]_{\bar{\partial}_{Y}\qquad \qquad \qquad}  & Y\tens C_{-1} (\C \mathbb{P}^{n-1}) \tens_{C(\C \mathbb{P}^{n-1})} \Omega^{0,1} (\C \mathbb{P}^{n-1})
	}
\end{align*} 
which happens because the $ L_{pr} $ maps commute with $ \theta $ in the formula (\ref{eqn43}).
\end{proof}

\end{document}